\documentclass[reqno]{amsart}
\usepackage[utf8]{inputenc} 
\usepackage{amsmath}
\usepackage{bbm}
\usepackage{mathtools}
\usepackage{hyperref}
\usepackage{amssymb}

\theoremstyle{definition}
\newtheorem{definition}{Definition}[section]

\theoremstyle{plain}

\newtheorem{theorem}{Theorem}

\newtheorem{lemma}{Lemma}

\newcommand\smin[3] 
 {\mathcal{S}_{#1}^{\min}(#2,#3)}
 \newcommand\snew[3] 
 {\mathcal{S}_{#1}^{\text{new}}(#2,#3)}
 \newcommand\sold[3] 
 {\mathcal{S}_{#1}^{\text{old}}(#2,#3)}
 \newcommand\sfull[3] 
 {\mathcal{S}_{#1}(#2,#3)}
 \newcommand\sexo[2]{\mathcal{S}_1^{\text{exotic}}(#1,#2)}
 \newcommand\sdih[2]{\mathcal{S}_1^{\text{dihedral}}(#1,#2)}

\newcommand\ebasis[1]{\mathcal{E}_1^\text{basis}(N,#1)}

\DeclareMathOperator{\sl2}{SL_2}

\DeclareMathOperator{\Ord}{Ord}
\DeclareMathOperator{\lcm}{LCM}
\DeclareMathOperator{\gal}{Gal}

\title{Computation of weight 1 modular forms with exotic representations}
\author{Kieran Child \\ University of Bristol}
\email{kieran.child@bristol.ac.uk}
\address{University of Bristol \\ Beacon House \\ Queens Road \\ Bristol \\ BS8 1QU \\ UK}

\begin{document}

\begin{abstract}
We present a deterministic algorithm for computing spaces of weight 1 modular forms with exotic representations. This algorithm is an improved version of Schaeffer's Hecke stability method, utilising the author's previous work on the twist-minimal trace formula for weight 2 holomorphic forms, and presenting a method of lifting forms from characteristic $p$.

The algorithm was used to compute all such forms with level at most 10,000. Together with Sutherland's computation of dihedral forms, this allows us to present the dimensions of all weight 1 newform spaces up to level 10,000.
\end{abstract}

\maketitle

\section{Introduction}
The Langlands program conjectures associations between cuspidal automorphic forms of $\text{GL}_n$ and irreducible, $n$-dimensional representations of the absolute Galois group $\gal(\overline{\mathbb{Q}}/\mathbb{Q})$. Following significant results in \cite{DeligneSerre} and \cite{KhareWintenberger}, this association has been proven for the 2-dimensional, odd case. We have a bijection:
\begin{equation}
\left\{ \begin{matrix}\text{Weight 1 newforms of}\\ \text{level $N$ and character $\chi$}\end{matrix}\right\} \longleftrightarrow \left\{ \begin{matrix}\text{Odd, irreducible, 2-dimensional} \\ \text{representations of $\gal(\overline{\mathbb{Q}}/\mathbb{Q})$ with} \\ \text{conductor $N$ and determinant character $\chi$}\end{matrix}\right\}.
\end{equation}
Hecke eigenvalues of the newform are equal to traces of Frobenius elements in the representation. Despite this bijection having been established, concrete examples are difficult to construct. This is especially so when the projective image of the representation is not dihedral.

Typical approaches to computing Fourier expansions of arbitrary newforms, such as by modular symbols (see \cite{Cremona}) or trace formulae (see \cite{Mypaper}) require that the weight be at least 2. Dimension formulae, such as in \cite{CohenOesterle}, also have this restriction. Instead, a number of specialised methods for computing weight 1 forms have been proposed. See \cite[Section 13.6.1]{CohenStromberg} for a review of three such methods. The most efficient in general is the Hecke stability method, presented in \cite{Schaeffer}.

While the computation of weight 1 newforms of a given level and character is slow, there are two related spaces which can be computed much quicker. Firstly, weight 1 newforms associated with representations with dihedral projective image can be computed via class field theory. Secondly, weight 1 newforms over characteristic $p$ can be computed quickly, by the algorithms in \cite{Edixhoven} or the $\mathbb{F}_p$ version of the Hecke stability method. These give lower and upper bounds respectively on the dimension of a weight 1 newform space of given level and character. A slower, deterministic computation of the space is then only required if these bounds are not equal. This approach was taken to compute all weight 1 forms up to level 1,500 as described in \cite{BuzzardComputation}, with the resulting data available at \cite{LauderTables}, and has since been performed up to level 4,000 for the LMFDB (see \cite{LMFDBdata}). From private discussion with Bill Allombert, we are aware that such a computation up to level 5,000 has also been performed.

This paper presents an improved version of the Hecke stability method to compute all weight 1 newforms up to a given level. Improvements come from utilising the twist-minimal trace formula presented in \cite{Mypaper} to compute weight 2 forms, and the introduction of a method of lifting and verifying forms from characteristic $p$, thus reducing the number of characteristic 0 calculations which need to be performed. The outcome of these improvements is the computation of all weight 1 newforms up to level 10,000, with this data available at \cite{wt1data}.

The structure of this paper is as follows. Section \ref{sec:computation} covers the preliminary theory used in our work, and presents our results. Section \ref{sec:proof} constitutes a proof of Theorem \ref{thm:MainTheorem}. Effectively, this is a proof that the proposed algorithm does indeed produce a basis of any specified weight 1 newform space. Finally, Section \ref{sec:commentary} discusses the practical computation, proving Theorem \ref{thm:data}, that our data up to level 10,000 is correct.

\subsection*{Acknowledgements}
This research is supported by EPSRC DTP EP/R513179/1 funding. I would like to thank Min Lee, Andrew Booker and Jonathan Bober for their guidance and suggestions with this work. I would also like to thank Andrew Sutherland for providing the dihedral data.

\section{Computation of weight 1 forms}
\label{sec:computation}
\subsection{Preliminary theory}
For any $N \in \mathbb{N}$, we define the congruence subgroup $\Gamma_0(N) \subset \sl2(\mathbb{Z})$ by the following condition:
\begin{equation}
\Gamma_0(N)=\left\{ \begin{pmatrix} a & b \\ c & d \end{pmatrix} \in \sl2(\mathbb{Z}): N\mid c \right\}.
\end{equation}

Let $\chi$ be a level $N$ Dirichlet character. That is, a homomorphism:
\begin{equation}
\chi : \left(\mathbb{Z}/N\mathbb{Z}\right)^* \rightarrow \mathbb{C},
\end{equation}
where the domain is extended to all $\mathbb{Z}$ by setting $\chi(n)=0$ whenever $(n,N)>1$. The conductor of $\chi$ is denoted $\mathfrak{f}(\chi)$. We define $\chi(\gamma)$ for $\gamma = \begin{pmatrix} a & b \\ c & d \end{pmatrix} \in \Gamma_0(N)$ by $\chi(\gamma):=\chi(d)=\overline{\chi(a)}$.

The group $\sl2(\mathbb{R})$ acts on the Poincar\'e upper half plane $\mathbb{H}=\{z \in \mathbb{C}:\Im(z)>0\}$ by M\"obius transformations:
\begin{equation}
\gamma z = \frac{az+c}{cz+d},\;\;\;\;\;\;\; \forall \; \gamma = \begin{pmatrix} a & b \\ c & d \end{pmatrix} \in \sl2(\mathbb{R}).
\end{equation}
Fixing $k \in \mathbb{N}$, we define the weight $k$ slash action of $\gamma = \begin{pmatrix} a & b \\ c & d \end{pmatrix} \in \sl2(\mathbb{R})$ on any complex-valued function $f$ as:
\begin{equation}
f|_k \gamma(z)=(cz+d)^{-k}f(\gamma z).
\end{equation}
If there exists a level $N$ Dirichlet character $\chi$ such that for all $\gamma \in \Gamma_0(N)$ we have:
\begin{equation}
f|_k\gamma(z)=\chi(\gamma)f(z),
\end{equation}
then we say that $f$ is weakly modular of weight $k$, level $N$ and character $\chi$. 

The set $\mathbb{Q} \cup \{i\infty\}$ is called the \emph{cusps} of $\sl2(\mathbb{Z})$. It is the $\sl2(\mathbb{Z})$-orbit of $i\infty$ under M\"obius transformations. If a function $f(z)$ is polynomially bounded as $z$ tends to any cusp, then we say $f$ is `holomorphic at the cusps.' This is equivalent to requiring $f|_k\gamma(z)$ be polynomially bounded as $\Im(z)\rightarrow \infty$ for all $\gamma \in \sl2(\mathbb{Z})$.

Fix $k, N \in \mathbb{N}$ and let $\chi$ be a level $N$ Dirichlet character. A weight $k$, level $N$, character $\chi$ modular form is a complex-valued function $f$ which has the following properties:
\begin{enumerate}
\item $f$ is holomorphic on $\mathbb{H}$.
\item $f$ is holomorphic at the cusps.
\item $f$ is weakly modular with weight $k$, level $N$ and character $\chi$.
\end{enumerate}

One can immediately construct examples of modular forms, called Eisenstein series, as follows. Fix a weight $k$, and let $\chi_1$ and $\chi_2$ be Dirichlet characters level $N_1$ and $N_2$ respectively such that $\chi_1\overline{\chi_2}(-1)=(-1)^k$. For $\Re(s)>1$ we define the series:
\begin{equation}
E_{\chi_1}^{\chi_2}(z;s) = \sum_{(c,d)=1} \frac{\Im(N_2z)^s \chi_1(c)\chi_2(d)}{|cN_2z+d|^{2s-k}(cN_2z+d)^{k}}
\end{equation}
If the weight $k$ is at least 3 then $E_{\chi_1}^{\chi_2}(z;k/2)$ is a level $N_1N_2$ modular form with character $\chi_1\overline{\chi_2}$. For weights 1 and 2 we obtain Eisenstein series by analytic continuation, and in all cases the resulting modular form is written $E_{\chi_1}^{\chi_2}(z)$. See \cite{Young} for details, including the derivation of explicit expressions for Fourier coefficients. For any given Eisenstein series $f$, there exists some cusp $c$ such that $f$ is not vanishing as $z \rightarrow c$.

In contrast, if $f$ is a modular form such that $f(z) \rightarrow 0$ as $z$ approaches any cusp then we say that $f$ is a cusp form. All modular forms (respectively cusp forms) of a given weight, level and character constitute a vector space which we denote $\mathcal{M}_k(N,\chi)$ (respectively $\sfull{k}{N}{\chi}$). The subspace of modular forms spanned by Eisenstein series is denoted $\mathcal{E}_k(N,\chi)$, and any given modular form space $\mathcal{M}_k(N,\chi)$ decomposes as $\sfull{k}{N}{\chi} \oplus \mathcal{E}_k(N,\chi)$. Given two modular forms $f$ and $g$, with at least one a cusp form, the Petersson inner product is defined as:
\begin{equation}
\langle f, g \rangle_N = \int_{z=x+iy \in \Gamma_0(N)\backslash \mathbb{H}} f(z)\overline{g(z)}y^{k}\frac{dxdy}{y^2}.
\end{equation}

With this inner product, the cusp form space of a given weight, level and character is a Hilbert space and is orthogonal to the Eisenstein space. From the definition of modularity, we see that if $f(z) \in \sfull{k}{N}{\chi}$ then $f(bz)\in\sfull{k}{M}{\chi}$ whenever $bN \mid M$. The subspace of $\sfull{k}{N}{\chi}$ spanned by forms arising in this manner from levels less than $N$ is called the oldform space and is denoted $\sold{k}{N}{\chi}$. The orthogonal complement of $\sold{k}{N}{\chi}$ with respect to the Petersson inner product is called the newform space and is denoted $\snew{k}{N}{\chi}$. Consequently, any cusp form space further decomposes into lifts from newform spaces. This decomposition was presented and studied in \cite{AtkinLehner}.

Modular forms are completely described by their Fourier expansions:
\begin{equation}
f(z)=\sum_{n \ge 0}a_n e^{2\pi i n z}.
\end{equation}
Note that $f \in \sfull{k}{N}{\chi}$ implies that $a_0=0$. Let $f_1$ and $f_2$ be modular forms of the same weight, level and character. The Sturm bound is a non-negative integer $B$ such that if $f_1 \neq f_2$ then the truncated Fourier expansions of $f_1$ and $f_2$ up to the $B$-th coefficient must also differ. The Sturm bound is given by:
\begin{equation}
B(\mathcal{M}_k(N,\chi))=\left \lfloor\frac{mk}{12} \right \rfloor,
\end{equation}
where $m$ is the index:
\begin{equation}
m = \left[\sl2(\mathbb{Z}):\Gamma_0(N)\right] = N \prod_{p\mid N}\left(1+\frac{1}{p}\right).
\end{equation}
This bound holds regardless of character (see \cite[Corollary 9.20]{Stein} for the weight 2 case and \cite[Lemma 5]{BuzzardComputation} for the weight 1 case). As such, computation of Fourier coefficients of a modular form up to this bound identifies the form, and a matrix of coefficients of basis elements, truncated at the Sturm bound, will have full rank.

In \cite{Young} we find the following construction of basis elements of the Eisenstein space, along with the Fourier expansions of these elements.
\begin{definition}
\label{eisbasisdef}
Fix a level $N$ and Dirichlet character $\chi$. Let $\ebasis{\chi}$ be the set of Eisenstein series $E_{\chi_1}^{\chi_2}(z)$ of weight 1 satisfying the criteria:
\begin{itemize}
\item $\chi_2(-1)=1$,
\item $\mathfrak{f}(\chi_1)\mathfrak{f}(\chi_2)\mid N$,
\item $\left(\mathfrak{f}(\chi_1),\frac{N}{\mathfrak{f}(\chi_1)}\right) \mid \frac{N}{\mathfrak{f}(\chi)}$,
\item $\chi_1 \overline{\chi_2} = \chi$.
\end{itemize}
\end{definition}
The set of $E_{\chi_1}^{\chi_2}(bz)$ for any $E_{\chi_1}^{\chi_2}(z) \in \ebasis{\chi}$ and $b \mid \frac{N}{\mathfrak{f}(\chi_1)\mathfrak{f}(\chi_2)}$ is a basis for $\mathcal{E}_1(N,\chi)$. Let $E_{\chi_1}^{\chi_2}(z)= \sum_{n \ge 0}a_n e^{2\pi i nz}$ be any such Eisenstein series, then $a_n$ is given by:
\begin{equation}
\label{eisensteinFourier}
a_n = \begin{cases}
\sum_{d \mid n}\chi_1(d)\chi_2(n/d) & \text{if }n>0, \\
\frac{-1}{2\mathfrak{f}(\chi)}\sum_{r<\mathfrak{f}(\chi)}r\chi(r) & \text{if }n=0 \text{ and }\chi_2=\mathbbm{1},\\
0 & \text{else.}
\end{cases}
\end{equation}

For certain newforms, the Fourier coefficients are completely determined by the eigenvalues of Hecke operators, which we now define as in \cite{IwaniecTopics}. Let $f$ be a weight $k$ modular form, and fix $n \in \mathbb{N}$ and a Dirichlet character $\chi$. The $n$-th Hecke operator $T_n^\chi(f)$ is given by:
\begin{equation}
T_n^\chi(f) = \frac{1}{n}\sum_{ad=n}\chi(a)a^k \sum_{b\text{ mod }d}f\left( \frac{az + b}{d}\right).
\end{equation}

All Hecke operators stabilise the spaces $\mathcal{M}_k(N,\chi),$ $\mathcal{E}_k(N,\chi)$, $\sfull{k}{N}{\chi}$, $\sold{k}{N}{\chi}$ and $\snew{k}{N}{\chi}$. We call a cusp form an \emph{eigenform} if it is an eigenfunction of all Hecke operators. In \cite{AtkinLehner} it is shown that any newspace $\snew{k}{N}{\chi}$ has a basis of eigenforms. These can be normalised such that $a_1=1$, whereupon the $n$-th Hecke eigenvalue is $a_n$. It is therefore these Hecke eigenvalues which we aim to compute.

Hecke relations (see \cite[Proposition 10.2.5]{CohenStromberg}) allow for the recovery of any $a_n$ given knowledge of $a_p$ for relevant primes $p$. Define $a_r = 0$ for any $r \not \in \mathbb{N}$. A normalised Hecke eigenform has $a_1=1$, and for any $n>1$ with prime $p \mid n$, the $n$-th Fourier coefficient of the eigenform is given by:
\begin{equation}
a_n = a_{\frac{n}{p}}a_p-\chi(p)p^{2k-1}a_{\frac{n}{p^2}}
\end{equation}

Computation of the $p$-th Hecke eigenvalue of an arbitrary weight 1 eigenform is, however, not an easy task. While a trace formula for all automorphic forms of weight 1 can be generated, it is not possible to extract the holomorphic contribution (see \cite[Chapter 10, Section 4]{HejhalTrace}). As such, computation of the $a_p$ has been performed a number of indirect ways (see \cite[Section 13.6.1]{CohenStromberg} for a review of three methods), with the most efficient being the \emph{Hecke stability method} presented in \cite{Schaeffer}. This is performed as follows:
\begin{enumerate}
\item Let $N$ and $\chi$ be the given level and character of the desired weight 1 newform space. Compute the space $S_2(N,\mathbbm{1})$.
\item Divide this by any Eisenstein series $E \in \mathcal{E}_1(N,\overline{\chi})$. Denote the result $\mathcal{M}_1^*(N,\chi)$.
\item Find the maximal subspace of $\mathcal{M}_1^*(N,\chi)$ which is stable under all Hecke operators. Denote this $\tilde{V}$.
\item We have $\sfull{1}{N}{\chi} \subset \tilde{V} \subset \mathcal{M}_1(N,\chi)$, and so removing the contribution of Eisenstein series and oldforms from $\tilde{V}$ gives $\snew{1}{N}{\chi}$.
\end{enumerate}
Practically, `computing a space' means computing a matrix of basis coefficients up to the Sturm bound. It is also shown in \cite{Schaeffer} that this method works (under additional constraints) for computing weight 1 forms over characteristic $p$. Computation over such fields is considerably faster than computation over $\mathbb{Q}$, and given that newforms over $\mathbb{Q}$ project onto $\mathbb{F}_p$ as newforms, one might hope that only the characteristic $p$ computation is needed. Unfortunately, there exist so-called `ethereal forms' (see \cite[Appendix A]{Edixhoven}) over characteristic $p$ which are not the projection of forms over $\mathbb{Q}$. Thus, restricting solely to computation over characteristic $p$ is an unreliable method for recovering forms over $\mathbb{Q}$.

We now move onto Artin representations. Weight 1 eigenforms are associated in an explicit fashion with Artin representations, and we use this association to deduce important results about their Hecke eigenvalues. Let $G_\mathbb{Q}$ be the absolute Galois group $\gal(\overline{\mathbb{Q}}/\mathbb{Q})$.
\begin{definition}
An $n$-dimensional Artin representation $\rho$ is a continuous homomorphism.
\begin{equation}
\rho : G_\mathbb{Q} \rightarrow \text{GL}_n(\mathbb{C}).
\end{equation}
\end{definition}
We can compose $\rho$ with $\det: \text{GL}_n(\mathbb{C}) \rightarrow \mathbb{C}^\times$ to give the `determinant character' $\chi$, a 1-dimensional Artin representation. Letting $\iota \in G_{\mathbb{Q}}$ be complex conjugation, we say that $\rho$ is odd if $\det(\rho(\iota))=-1$ and $\rho$ is even if $\det(\rho(\iota))=1$. An Artin representation has the same parity as its determinant character.

Continuity implies that $\rho$ filters through a finite extension $K/\mathbb{Q}$, and thus the image of $\rho$ under projection onto $\text{PGL}_n(\mathbb{C})$ is a finite subgroup of $\text{PGL}_n(\mathbb{C})$. We categorise 2-dimensional Artin representations by their image under this projection, which must be either cyclic, dihedral, or  one of $A_4, S_4$ or $A_5$. If $\rho$ is irreducible, then the cyclic case is impossible. If the image of an irreducible representation is not dihedral, then we say it is exotic.

We follow \cite{Martinet} to define the $L$-function attached to an Artin representation $\rho$. Let $K$ be the minimal number field through which $\rho$ filters. For any prime $p$, fix a prime $\mathfrak{p}$ lying above $p$, with inertia group $I_{\mathfrak{p}}$. Let $V^{I_{\mathfrak{p}}}$ be the fixed subspace:
\begin{equation}
V^{I_\mathfrak{p}} = \left\{ x \in \mathbb{C}^n : \sigma(x)=\sigma\;\; \forall \; \sigma \in \rho(I_\mathfrak{p})\right\}.
\end{equation}
Let $Frob_{\mathfrak{p}}$ be any Frobenius element for the ideal $\mathfrak{p}$. Then the determinant:
\begin{equation}
\det\left(1-p^{-s}\rho|_{V^{I_\mathfrak{p}}}(Frob_\mathfrak{p})\right),
\end{equation}
is seen to be well-defined and irrespective of the choice of $\mathfrak{p}$. We define the Artin $L$-function attached to $\rho$ by the Euler product:
\begin{equation}
L(\rho,s) = \prod_p \det\left(1-p^{-s}\rho|_{V^{I_\mathfrak{p}}}(Frob_\mathfrak{p})\right)^{-1}.
\end{equation}
We state the association with eigenforms via this $L$-function. One can construct an $L$-function from any eigenform $f$ by the product:
\begin{equation}
L(f,s) = \sum_n \frac{a_n}{n^s} = \prod_{p \mid N} \left(1-a_pp^{-s}\right)^{-1} \prod_{p \nmid N} \left(1-a_pp^{-s}+\chi(p)p^{k-1-2s}\right)^{-1}.
\end{equation}
By \cite[Theorem 4.1]{DeligneSerre}, all such $L$-functions for weight 1 eigenforms are equivalent to $L$-functions arising from Artin representations, and thus all weight 1 eigenforms have an associated Artin representation. In \cite{KhareWintenberger} it is proven that the converse is also true. Consequently, we use the $L(\rho,s)$ construction to deduce information about $L(f,s)$ and in turn $a_p$.

We define the Satake parameters at $p$ as the eigenvalues of $\rho|_{V^{I_\mathfrak{p}}}(Frob_\mathfrak{p})$. As these are all matrices of finite order, Satake parameters are roots of unity. For an unramified prime, the inertia group is trivial and so we have two (not necessarily distinct) Satake parameters. For a ramified prime, we have zero or one satake parameters, corresponding to the dimension of $V^{I_{\mathfrak{p}}}$.

The specific $L$-function association given in \cite{DeligneSerre} includes that a prime $p$ is unramified in $K$ if and only if $p \nmid N$ where $N$ is the level of the associated eigenform $f$. Thus we can equivalently define Satake parameters in terms of the Fourier expansion of $f$:
\begin{itemize}
\item For $p \nmid N$, the roots of $x^2-a_px+\chi(p)$ are the Satake parameters at $p$.
\item For $p \mid N$ with $a_p=0$, there are no Satake parameters at $p$.
\item For $p \mid N$ with $a_p \neq 0$, $a_p$ is the Satake parameter at $p$.
\end{itemize}
From \cite[Theorem 3]{Li}, we see that:
\begin{equation}
\label{apis0}
p^2 \mid N,\; \nu_p(\mathfrak{f}(\chi))<\nu_p(N) \iff a_p=0.
\end{equation}
Therefore, we can compute $a_p$ for any prime $p$ given $N$, $\chi$ and a Satake parameter at $p$ (when one exists). As an immediate application, \cite[Theorem 3]{Li} shows that when $\nu_p(N)=1$ and $\chi_p=\mathbbm{1}$ we must have $|a_p|=\sqrt{p}$, but as the Satake parameters are roots of unity, this cannot be the case, and so $\snew{1}{N}{\chi}=0$. In particular, there are no non-zero weight 1 newforms for any level $N \equiv 2 \pmod{4}$.

The subspaces of $\snew{1}{N}{\chi}$ spanned by eigenforms associated with Artin representations with dihedral or exotic projective image are denoted $\sdih{N}{\chi}$ and $\sexo{N}{\chi}$ respectively. Computation of the dimensions of $\sdih{N}{\chi}$, along with projective images of forms in these spaces, can be performed efficiently via class field theory, as explained in \cite{BuzzardTheory}. Such a computation, giving the dimensions and specific projective images of all weight 1 dihedral forms, has been performed up to level 40,000 and detailed in \cite{LMFDBpaper}. We assume this information for our computation and so reduce our scope to the computation of $\sexo{N}{\chi}$.
\subsection{Results}
Our computation is primarily the result of improvements to the Hecke stability method. Firstly, we describe an efficient computation of $\sfull{2}{N}{\mathbbm{1}}$ coming from the twist-minimal trace formula given in \cite{Mypaper}. 
\begin{definition}\label{def:twistpairs}
Fix a level $N \in \mathbb{N}$. A \emph{twist pair} $\langle M, \psi \rangle$ is a tuple of $M \mid N$ and Dirichlet character $\psi$ satisfying the following conditions:
\begin{itemize}
\item For all $p$ with $\psi_p \neq \mathbbm{1}$, $\nu_p(N)=2\nu_p(\mathfrak{f}(\psi))$;
\item For all $p>2$ either $\nu_p(M)=\nu_p(N)-\nu_p(\mathfrak{f}(\psi))$, or $\nu_p(M)=0$ and $\Ord(\psi_p)=2$;
\item If $2 \mid \nu_2(N)$ and $\nu_2(N)>3$ then $\nu_2(M) \in \left\{\frac{\nu_2(N)}{2}-1, \nu_2(N)-1, \nu_2(N)-2\right\}$;
\item If $\psi_2 = \mathbbm{1}$ then $\nu_2(M)=\nu_2(N)$.
\end{itemize}

\end{definition}
Note that if $\psi_1$ and $\psi_2$ are Galois-conjugate characters, and $\langle M,\psi_1 \rangle$ is a twist pair, then so is $\langle M, \psi_2 \rangle$. We can thus associate any Galois orbit with a Dirichlet character $\Psi$ defined over the cyclotomic field $\mathbb{Q}(\zeta_r)$ where $r$ is the order of the orbit. Any twist pair $\langle M, \psi \rangle$ with $\psi$ in the orbit then arises as an embedding of $\langle M, \Psi \rangle$ into $\mathbb{C}$.

For any twist-pair $\langle M, \psi \rangle$ we define the trace form $\mathcal{T}_{M,\psi}$ by the Fourier expansion $\sum_n a_n e^{2 \pi i z n}$ where the $a_n$ are given as follows. If $(n^2,N)$ is not squarefree then $a_n=0$. Otherwise:
\begin{equation}
a_n=(C_1-C_2-C_3+C_4)\prod_{\substack{p|\mathfrak{f}(\psi)\\p\nmid M}}\psi_p(n),
\end{equation}
with:
\begin{equation}
\label{wt2formula}
\begin{split}
C_1&=\frac{\underset{n\text{ is square}}{\delta}}{12} \prod_{p|M}\begin{cases}
p^e+p^{e-1}&\text{if }s=e,\\
\frac{\phi(\lceil p^{e-2}\rceil)(p-1)}{1+\underset{\substack{2|e \\p>2}}{\delta}}(1+\underset{e>1}{\delta}p-\underset{e=2}{\delta}2)&\text{else.}
\end{cases}\\
C_2&=\sum_{t^2<4n}\frac{h(d)}{w(d)}\prod_{\substack{p|\ell \\ p\nmid M}}S_p(1,\mathbbm{1},t,n)\prod_{p|M}S_p^{\min}(p^e,\psi_p^2,t,n)_{\overline{\psi_p}}\\
C_3&=\sum_{\substack{d|n \\ d\le \sqrt{n}}}'d \prod_{p|M}\begin{cases}
\Re\left(\psi\left(\frac{d^2}{n}\right)\right)&\text{if }s=e,\\
0 & \text{else.}
\end{cases}\\
C_4&=\underset{\psi^2=\mathbbm{1}}{\delta}\mu(M)\prod_{\substack{p|n\\ p\nmid M}}\sigma(p^{\nu_p(n)})\\
\end{split}
\end{equation}
In all terms, $s= \nu_p(\mathfrak{f}(\psi^2))$, $e=\nu_p(M)$ and $\delta$ is the characteristic function of the condition underneath it. In $C_1$, $\phi(a)$ is Euler's totient function, the size of $(\mathbb{Z}/a\mathbb{Z})^*$. In $C_2$, we define $d$ as the fundamental discriminant of $t^2-4n$, with $\ell \in \mathbb{N}$ such that $t^2-4n=d\ell^2$. The functions $h(d)$ and $w(d)$ are the class number and roots of unity of a quadratic extension with fundamental discriminant $d$. In $C_3$, the dash on the sum indicates that if $d=\sqrt{n}$ an extra factor of $\frac{1}{2}$ is present, and we define $\Re(\chi(x))$ as $\chi(x)+\chi\left(x^{-1}\right)$ when $x$ is invertible, and $0$ otherwise. The functions $\mathcal{S}_p$ and $\mathcal{S}_p^{\min}$ present in $C_2$ are defined as follows. Set $v = \nu_p(t^2-4n)$, and let $\left(\frac{\cdot}{p}\right)$ denote the Kronecker symbol. For $p>2$ we define:
\begin{equation}
S_p^{\min}(p^e,\psi_p,t,n)_{\overline{\psi}}=\begin{cases}
2p^{\nu_p(\ell)}+\left(1-\left(\frac{d}{p}\right)\right)\frac{2p^{\nu_p(\ell)}-p^e-p^{e-1}}{p-1} & \text{if }s>0,v \ge 2e-1\\ \\
\underset{\left(\frac{d}{p}\right)=1}{\delta}p^{\nu_p(\ell)}\Re \left(\psi \left(\frac{t(t+u)}{2n}-1\right)\right) & \text{if }s>0,v<2e-1\\ \\
\left(\left(\frac{d}{p}\right)-1\right)\left(\frac{n}{p}\right)& \text{if }s=0,p|\mathfrak{f}(\psi)\\ \\
\underset{\substack{e=1 \\ \text{or }\left(\frac{n}{p}\right)=1}}{\delta}\left(1-\left(\frac{d}{p}\right)\right) \frac{p^{e-3}}{(2,e)}\\
\cdot (\underset{e>2}{\delta}+p(\underset{e=2}{\delta}+\underset{\substack{2|e\\ v=e-2}}{\delta}-\underset{v \ge e-1}{\delta}p)) & \text{if }v\ge e-2, p\nmid \mathfrak{f}(\psi)\\ \\
0 & \text{else}
\end{cases}
\end{equation}
For $p=2$ with $s=e$ we define:
\begin{equation}
S_2^{\min}(2^e,\psi^2,t,n)_{\overline{\psi}}=\begin{cases}
(1-\underset{v=2e}{\delta}2)\\
\cdot \left( (2^{\lfloor \frac{v}{2}\rfloor+1}-3\cdot 2^{e-1})(1-\left(\frac{d}{2}\right))+\underset{2\nmid d}{\delta}2^{\nu_2(\ell)+1}\right) & \text{if }v \ge 2e\\
\underset{\left(\frac{d}{2}\right)=1}{\delta}2^{\nu_2(\ell)}\Re\left(\psi\left(\frac{t(t+u)}{2n}-1\right)\right)&\text{if }v <2e-1\\
\end{cases}
\end{equation}
For $p=2$ with $2|\mathfrak{f}(\psi)$ but $s<e$ we define:
\begin{equation}
S_2^{\min}(2^e,\mathbbm{1},t,n)_{\overline{\psi}}=\left(1-\left(\frac{d}{2}\right)\right)\lceil 2^{e-3}\rceil \begin{cases}
-3 & \text{if }v >e\ge 3\\
\psi(n)(2(-1)^d-1)& \text{if }v \in \{e,e-1\}, e\ge 3, s=0\\
(-1)^e+2 & \text{if }s>0, v=e \\
1-2(-1)^d & \text{if }s>0,v=e-1\\
\psi(n)\left(\underset{\substack{e=2\\ v=0}}{\delta}\frac{3}{2}-1\right)&\text{if }e\in \{1,2\}\\
0 &\text{else}
\end{cases}
\end{equation}
In all these cases, $u \equiv \ell \sqrt{d}$ is defined mod $p^{e+2}$ if $p=2$ and mod $p^e$ otherwise. Finally, we define:
\begin{equation}
S_p(1,\mathbbm{1},t,n) = p^{\nu_p(\ell)}+\left(1-\left(\frac{d}{p}\right)\right)\frac{p^{\nu_p(\ell)}-1}{p-1}.
\end{equation}
We can construct a polynomial traceform $\mathcal{T}_{M,\Psi}$ by this formula, evaluating the character $\Psi$ associated with the Galois orbit of some twist pair $\langle M, \psi \rangle$.

\begin{lemma}
\label{s2computation}
For any polynomial traceform $\mathcal{T}_{ M,\Psi }$ associated to the Galois orbit containing $\langle M, \psi \rangle$, denote $r = \Ord(\psi)$ and let $\beta$ be the Sturm bound at weight 2 and level $M$. The first coefficient of $\mathcal{T}_{ M, \Psi }$ is a non-negative integer, which we denote $d$. There exist $A$ and $B$ finite subsets of $\mathbb{N}$, with $a \le \frac{r}{2}$ for all $a \in A$, $b \le \beta$ for all $b \in B$, and $|B|=d$, such that embeddings given by $\zeta_r \rightarrow e^{2 \pi i a/r}$ of elements $T_{b}\mathcal{T}_{M,\Psi}$ across all $a \in A, b \in B$ and Galois orbits of twist pairs $\langle M, \psi \rangle$ constitutes a basis for $\snew{2}{N}{\mathbbm{1}}$.
\end{lemma}
The sets $A$ and $B$ are difficult to write down directly, but in practice, easy to generate iteratively for any given case. To find $B$ we check each possible natural number $b$ in turn, appending to the set if $T_{b}\mathcal{T}_{M,\Psi}$ is not some linear combination of the established $T_{b'}\mathcal{T}_{M,\Psi}$ for $b' \in B$, until we have $d$ linearly independent elements.

To find $A$ we proceed as follows. We say that $\psi_1$ and $\psi_2$ are twist-equivalent if for every prime $p$ we have $\psi_{1,p} \in \{\psi_{2,p}, \overline{\psi_{2,p}}\}$. For a twist pair $\langle M, \Psi \rangle$ we check each possible natural number $a$ in turn, appending to $A$ if there does not already exist a twist pair $\langle M, \Psi' \rangle$ with $Ord(\Psi)=Ord(\Psi')=r$ and embedding $a'$ such that the character $\psi$ given by embedding $\Psi$ into $\mathbb{C}$ by $\zeta_r \rightarrow e^{2 \pi i a/r}$ is twist-equivalent to the character $\psi'$ given by embedding $\Psi'$ into $\mathbb{C}$ by $\zeta_r \rightarrow e^{2 \pi i a'/r}$.

Given the efficient computation of arbitrary coefficients of a basis for $\snew{2}{N}{\mathbbm{1}}$, we present an algorithm for computing $\snew{1}{N}{\chi}$ from this basis, assuming knowledge of the dimensions and projective images of $\sdih{N}{\chi}$, and with the set $\ebasis{\chi}$ as in Definition \ref{eisbasisdef}. This is essentially the Hecke stability method with practical improvements.

\begin{enumerate}
\item Let $N$ and $\chi$ be the given level and character of the desired weight 1 newform space. Let $D$ be the LCM of all orders of projective images of Artin representations associated with eigenbasis elements for $\sdih{N'}{\chi}$ where $N'$ ranges across all levels satisfying $\mathfrak{f}(\chi)\mid N' \mid N$, and let $H=2\lcm(60,\phi(N),D)$, where $\phi$ is Euler's totient function. Fix the least primes $p$ and $q$ satisfying $p \nmid N$ and $q \equiv 1 \pmod{H}$.
\item Let $d=\dim(\sfull{2}{N}{\mathbbm{1}})$. Compute a $m\times d$ basis matrix $M$ for $\sfull{2}{N}{\mathbbm{1}}$ using Lemma \ref{s2computation}, where $m$ is at least the Sturm bound of $\sfull{1}{N}{\chi}$ and is such that the submatrix of the first $m/p$ rows has rank $d$.
\item Fix a ring homomorphism $\sigma : \mathbb{Z}[\zeta_H] \rightarrow \mathbb{F}_q$ whose restriction to $\langle \zeta_H \rangle$ is injective, and use this to project $M$ and $\ebasis{\overline{\chi}}$ onto $\mathbb{F}_q$.
\item Let $M/E$ for an Eisenstein series $E$ be the $m\times d$ matrix given by the leading $m$ coefficients of each basis element in $M$ divided by $E$. Working over $\mathbb{F}_q$, find $M^*$, the intersection of $M/E$ across all $E \in \ebasis{\overline{\chi}}$.
\item Find $\tilde{V}$, the maximum $T_p$-stable subspace of $M^*$. Diagonalise $\tilde{V}$, removing the contribution from $\sold{1}{N}{\chi}$ in the process. The details of this diagonalisation will be given in the next section.
\item For each remaining element $\tilde{f}$ in $\tilde{V}$ and each prime $p \le m$, compute a (non-zero) root $\alpha_{p,\mathbb{F}_q}$ of $x^2-a_px+\chi(p)$ over $\mathbb{F}_q$. If there is some $p$ where this polynomial doesn't split over $\mathbb{F}_q$ then discard $\tilde{f}$.
\item For all primes $p \le m$ let $\alpha_{p,\mathbb{Q}}$ be the root of unity which is the preimage of $\alpha_{p,\mathbb{F}_q}$ under $\sigma$. Define a lifted $f$ over $\mathbb{Z}[\zeta_H]$ by setting $a_p=\alpha_{p,\mathbb{Q}}+\chi(p)\overline{\alpha_{p,\mathbb{Q}}}$ and applying Hecke relations to recover all $a_n$.
\item Multiply $f$ over $\mathbb{Q}$ by any element in $\ebasis{\overline{\chi}}$. If the result is in $\sfull{2}{N}{\mathbbm{1}}$ then include $f$ in the output.
\end{enumerate}

\begin{theorem}
\label{thm:MainTheorem}
The output of the preceding algorithm is an eigenbasis for $\snew{1}{N}{\chi}$.
\end{theorem}

The computation of these eigenbases can be performed for each level in turn, so that the computed basis of $\sfull{2}{N}{\mathbbm{1}}$ can be reused for each character, and the elements in $\snew{1}{M}{\chi}$ can be lifted to $\sold{1}{N}{\chi}$ for every $M \mid N$ in step 5. Not every character and level needs to be computed, which motivates the following definition.

\begin{definition}
\label{admissible}
Fix $N \in \mathbb{N}$ and a level $N$ Dirichlet character $\chi$. We say this pair is \emph{admissible} if the following criteria are all met:
\begin{itemize}
\item The character $\chi$ is odd;
\item For any $p$ with $\nu_p(N)=1$, $\chi_p \neq \mathbbm{1}$;
\item For any $p \mid N$ with $\nu_p(N)=\nu_p(\mathfrak{f}(\chi))$, $\Ord(\chi_p) \in \left\{2, 3, 4, 5 \right\}$;
\item For any two such primes $p$ and $q$, $\Ord(\chi_p\chi_q) \neq 20$.
\end{itemize}
\end{definition}

\begin{lemma}
\label{admissiblelemma}
Fix $N_{\max} \in \mathbb{N}$, the maximum level to compute all weight 1 exotic forms up to. For any given $N \le N_{\max}$ and level $N$ Dirichlet character $\chi$, let $M$ be the least multiple of $N$ such that $M$ and $\chi$ are admissible. If $M>N_{\max}$, or if no such $M$ exists, then $\snew{1}{N}{\chi}$ does not need to be computed. Computing all other $\snew{1}{N}{\chi}$ with $N \le N_{\max}$ provides the dimensions (and Fourier coefficients) of all exotic forms with level at most $N_{\max}$.
\end{lemma}

We performed such a computation up to level $10,000$. This took 832 minutes on 192 processors, implemeted in PARI (see \cite{PARI2}) and performed in parallel using Parallel (see \cite{parallel}).

\begin{theorem}
\label{thm:data}
The dimensions of all spaces of newforms $\snew{1}{N}{\chi}$ for $N$ up to level 10,000, and the coefficients of all newforms in spaces where $\sexo{N}{\chi} \neq 0$, up to at least the Sturm bound, are as given at \cite{wt1data}.
\end{theorem}

\section{Proof of Theorem 1}
\label{sec:proof}
The formula for $\mathcal{T}_{M,\chi}$ preceding Lemma \ref{s2computation} is a specific application of the twist-minimal trace formula from \cite{Mypaper}. Let $f \in \snew{k}{N}{\chi}$ be a newform with Fourier expansion $\sum a_n e^{2 \pi i n z}$. The twist of $f$ by $\psi$ is defined as:
\begin{equation}
f_{\psi}(z) = \sum_{n>0}a_n \psi(n) e^{2 \pi i n z}.
\end{equation}
This is a cusp form of potentially different level to $N$, and character $\chi \psi^2$. As in the decomposition of cusp forms into lifts from spaces of newforms, we define spaces of twist-minimal forms as the orthogonal complement to the space spanned by forms arising as twists from lower levels. Spaces of newforms consequently decompose into twists from twist-minimal spaces, denoted $\smin{k}{N}{\chi}$. This decomposition only includes those twist-minimal spaces with twist-minimal characters, defined as follows:
\begin{definition}
Let $\chi$ be a level $N$ Dirichlet character. We say that $\chi$ is twist-minimal if for every $p \mid N$ at least one of the following holds:
\begin{itemize}
\item $\chi_p$ is primitive;
\item $p>2$ and $\chi_p=\mathbbm{1}$;
\item $p>2$ and $\text{Ord}(\chi_p)=2^{\nu_2(p-1)}$;
\item $p=2$ and $\nu_p(\mathfrak{f}(\chi))= \lfloor \frac{\nu_p(N)}{2} \rfloor$;
\item $p=2$, $\nu_p(\mathfrak{f}(\chi))=2$, $\nu_p(N)>3$ and $2\nmid \nu_p(N)$;
\item $p=2$, $\chi_p=\mathbbm{1}$, and $2 \nmid \nu_p(N)$ or $\nu_p(N)=2$.
\end{itemize}
\end{definition}

The definition of twist pairs in Definition \ref{def:twistpairs} are seen to be such that $\langle M, \psi \rangle$ is a twist pair if and only if $\overline{\psi^2}$ is a level $M$ twist-minimal character. An explicit trace formula for twist-minimal spaces with twist-minimal character is given in \cite[Theorem 1]{Mypaper}.

If $\smin{2}{M_1}{\overline{\psi_1^2}}_{\psi_1}=\smin{2}{M_2}{\overline{\psi_2^2}}_{\psi_2}$ then we say the pairs $\langle M_1, \psi_1 \rangle$ and $\langle M_2, \psi_2 \rangle$ are twist-equivalent. We see from \cite[Theorem 2]{Mypaper} that a basis for $\snew{2}{N}{\mathbbm{1}}$ is given by bases of $\smin{2}{M}{\overline{\psi^2}}_\psi$ where $M$ and $\psi$ range across twist pairs, excepting those pairs which are twist-equivalent.

Following \cite[Theorem 1]{Mypaper} we see that the $n$-th coefficient of the trace form of $\smin{2}{M}{\overline{\psi^2}}_{\psi}$ is given by:
\begin{equation}
a_n (C_1-C_2-C_3+C_4) \prod_{p \mid \mathfrak{f}(\psi)}\psi_p(n),
\end{equation}
where $C_i$ are the components of the twist minimal trace formula evaluated for $\smin{2}{M}{\overline{\psi^2}}$. Taking the evaluation of $\psi$ for factors $p \mid M$ into the sum gives the following expression for $C_1$:

\begin{equation}
C_1 = \frac{\psi(n) \overline{\psi^2(\sqrt{n})}}{12} \prod_{p \mid M} \begin{cases} p^e + p^{e-1} & \text{if }s=e,\\
\frac{\phi(\lceil p^{e-2} \rceil)(p-1)}{1+\underset{\substack{2\mid e \\ p>2}}{\delta}}+\underset{e>1}{\delta}p+\underset{e=2}{\delta}(2s-2)) & \text{else,}
\end{cases}
\end{equation}
where $s=\nu_p(\mathfrak{f}(\psi^2))$, $e=\nu_p(M)$ and $\psi^2(\sqrt{n})=0$ if $n$ is not square. We see from the definition of twist pairs that $e=2$ and $s \neq e$ implies $s=0$, and the evaluations of $\psi$ cancel, giving $C_1$ as in (\ref{wt2formula}). The remaining terms come from similar direct evaluations of the twist-minimal trace formula.

We note that each term in the formula is only dependent on evaluations of $\psi_p$, its level and conductor, and whether or not $\psi_p$ is quadratic. These are equal for every character in the Galois orbit of $\psi$, and therefore any $\mathcal{T}_{M,\psi}$ is the embedding of the polynomial traceform $\mathcal{T}_{M,\Psi}$ attached to the Galois orbit containing $\psi$.

As Galois conjugate spaces are the same dimension, the first coefficient of $\mathcal{T}_{M,\Psi}$ is an integer $d$ giving this dimension. As each $T_m\mathcal{T}_{M,\psi}$ is also conjugate, we can produce (by \cite[Theorem 2]{Mypaper}) a rank $d$ matrix consisting of leading coefficients of $T_m \mathcal{T}_{M,\Psi}$ for various $m$, whereupon a basis for $\smin{2}{M}{\overline{\psi^2}}_{\psi}$ is given by embedding this matrix into $\mathbb{C}$ with the map which sends $\Psi$ to $\psi$.

Taking this computation across all twist-pairs will count twists from twist-equivalent spaces multiple times, and so we use the following lemma to remove duplicates.

\begin{lemma}
Suppose $\chi$ is twist-equivalent to $\psi$. Let $\chi'$ be some character in the Galois orbit of $\chi$. Then there exists some $\psi'$ in the Galois orbit of $\psi$ which is twist-equivalent to $\chi'$.
\end{lemma}
\begin{proof}
For each prime we have $\chi_p'=\chi_p^n$. Define $\psi_p'=\psi_p^n$. As $\chi$ is twist-equivalent to $\psi$ we must have $\psi_p \in \{\chi_p,\overline{\chi_p}\}$. Thus, $\psi_p' \in \{\chi_p^n, \overline{\chi_p^n}\}$ for all $p$, and so $\psi'$ is twist-equivalent to $\chi'$.
\end{proof}
This means that twist-equivalence is an equivalence relation on Galois orbits as well as individual characters. If we compute $\mathcal{T}_{M,\Psi}$ for all Galois orbits, extract a full-rank basis matrix from this, and embed the results into all $e^{2 \pi i a/r}$ $a \in A$, the set of maximal distinct embeddings, then we recover a basis for $\snew{2}{N}{\mathbbm{1}}$ without duplication.

The remainder of this section constitutes the proof of Theorem \ref{thm:MainTheorem}.

\begin{lemma}
\label{lemma:eisintersect}
Let $M^*$ be the intersection of $\sfull{2}{N}{\mathbbm{1}}/E$ for each $E \in \ebasis{\overline{\chi}}$. Fix a prime $p \nmid N$. Then the maximal $T_p$-stable subspace of $M^*$, denoted $\tilde{V}$, is $\sfull{1}{N}{\chi}$.
\end{lemma}
\begin{proof}
In \cite{Schaeffer} it is proven that, for any Eisenstein series $E \in \mathcal{E}_1(N,\overline{\chi})$ and $p \nmid N$, the maximal $T_p$-stable subspace of $\sfull{2}{N}{\mathbbm{1}}/E$ is a subspace of $\mathcal{M}_1(N,\chi)$ containing $\sfull{1}{N}{\chi}$. Thus:
\begin{equation}
\sfull{1}{N}{\chi} \subset \tilde{V} \subset \mathcal{M}_1(N,\chi).
\end{equation}

In \cite{Young}, $\ebasis{\chi}$, along with lifts thereof, is shown to span $\mathcal{E}_1(N,\chi)$.  For any $E \in \ebasis{\overline{\chi}}$, the existence of some cusp at which $E$ is non-zero gives:
\begin{equation}
E(Bz)\overline{E(z)} \not \in \sfull{2}{N}{\mathbbm{1}}.
\end{equation}
Consequently, $M^* \cap \mathcal{E}_1(N,\chi) = 0$ and so $\tilde{V}=\sfull{1}{N}{\chi}$. 
\end{proof}

We now show that any form over $\mathbb{Q}$ will project onto $\mathbb{F}_q$ for a carefully chosen prime $q$ in such a way that the original form can be recovered. We first note that $\tilde{V}$ can be computed over $\mathbb{F}_q$, and so we acquire a space containing the projection of $\sfull{1}{N}{\chi}$.

\begin{lemma}
Define $H$ as in Theorem \ref{thm:MainTheorem}. Fix a prime $q$ such that $q \equiv 1 \pmod{H}$ and $q \nmid N$. Fix a projection $\mathbb{Z}[\zeta_H] \rightarrow \mathbb{F}_q$, then the computation of $\sfull{2}{N}{\mathbbm{1}}$, $\ebasis{\overline{\chi}}$, and $\tilde{V}$ are all well-defined over $\mathbb{F}_q$.
\end{lemma}
\begin{proof}
First, we show that all coefficients in the basis of $\sfull{2}{N}{\mathbbm{1}}$ are in $\mathbb{Z}[\zeta_H]$. This comes from the fact that trace form coefficients for $\sfull{k}{N}{\chi}$ are in $\mathbb{Z}[\chi]$ (see \cite[Theorem 28.4]{KnightlyLi}). By \cite[(30) and (34)]{Mypaper}, we can write $\text{Tr}T_n|\smin{k}{N}{\chi}$ as a combination of elements $\text{Tr}T_n|\smin{k}{N'}{\psi}$, where all $\psi$ have order dividing $\phi(N)$.

The Fourier expansion of Eisenstein series in (\ref{eisensteinFourier}) is similarly given in terms of evaluations of characters with conductor dividing $N$ and thus order dividing $\phi(N)$, and so are in $\mathbb{Z}[\zeta_H]$. Finally, the vector representing action by the $T_p$ operator on the Fourier expansion of any modular form has coefficients in $\mathbb{Z}[\chi]$, and so the computation of $\tilde{V}$ is well-defined over $\mathbb{F}_q$.
\end{proof}
The resultant matrix $\tilde{V}$ will clearly contain the projection of $\sfull{1}{N}{\chi}$ onto $\mathbb{F}_q$ but may also contain ethereal forms, and modular functions. These are subsequently removed by diagonalisation and lifting. We will show that forms in $\snew{1}{N}{\chi}$ are recovered by this process. We need the following lemma:
\begin{lemma}
\label{multdivlemma}
If $\alpha$ and $\beta$ are roots of unity, with $\Ord \left(\frac{\alpha}{\beta}\right) \mid a$ and $\Ord(\alpha\beta) \mid b$, then $\Ord(\alpha) \mid 2\lcm(a,b)$.
\end{lemma}

\begin{proof}
By assumption, $\alpha^a = \beta^a$, and thus $(\alpha\beta)^a = \alpha^{2a}$. Let $h=2\lcm(a,b)$, then $2a \mid h \implies \alpha^h=(\alpha\beta)^{\frac{h}{2}}$ and as $b \mid \frac{h}{2}$ we conclude $\alpha^h=1$.
\end{proof}

Using this, we restrict the possible orders of Satake parameters of weight 1 newforms.

\begin{lemma}
\label{lemma:satakeorder}
Fix $f$ a weight 1 eigenform of level $N$ and character $\chi$. Let $D$ be the LCM of all orders of projective images of Artin representations associated with eigenbasis elements for $\sdih{M}{\chi}$ where $M$ ranges across all values satisfying $\mathfrak{f}(\chi)\mid M \mid N$, and let $H=2\lcm(60,\phi(N),D)$. Any Satake paramter of $f$ has order dividing $H$.
\end{lemma}
\begin{proof}
The action of $\text{GL}_2(\mathbb{C})$ by conjugation on $2\times2$ matrices with trace 0 gives an isomorphism $\text{PGL}_2(\mathbb{C}) \cong \text{SO}_3(\mathbb{C})$. Thus, the projection of a 2-dimensional Artin representation gives a 3-dimensional representation of the projective image.

For a matrix $M_2 \in \text{GL}_2(\mathbb{C})$ the characteristic polynomial is
\begin{equation}
\label{charpolyM}
x^2 - \text{tr}(M_2)x + \det(M_2),
\end{equation}
and the characteristic polynomial of the associated matrix $M_3 \in SO_3(\mathbb{C})$ is:
\begin{equation}
\label{charpolyM3}
x^3 + \left(1-\frac{\text{tr}(M_2)^2}{\det(M_2)}\right)(x^2-x)- 1.
\end{equation}
Supposing the roots of (\ref{charpolyM}) are $\alpha$ and $\beta$ then the roots of (\ref{charpolyM3}) are $\frac{\alpha}{\beta}, \frac{\beta}{\alpha}$ and 1.

Let $\rho$ be a two-dimensional Artin representation with projective image $D_n$, and fix a prime $p \nmid N$. The above isomorphism associates with $\rho$ a 3-dimensional representation of $D_n$ with character value $\frac{\alpha}{\beta}+\frac{\beta}{\alpha}+1$ at $\overline{\rho}(p)$. The representation theory of dihedral groups gives finite possibilities for this character value, and so restricts the possible order of the root of unity $\frac{\alpha}{\beta}$. If $\frac{\alpha}{\beta}+\frac{\beta}{\alpha}+1=3$ then $\frac{\alpha}{\beta}=1$. If $\frac{\alpha}{\beta}+\frac{\beta}{\alpha}+1=1$ then $\Ord\left(\frac{\alpha}{\beta}\right)=4$, and if $\frac{\alpha}{\beta}+\frac{\beta}{\alpha}+1=1+2\cos(2\pi i k/n)$ for some $k\in \mathbb{N}$ then $\Ord\left(\frac{\alpha}{\beta}\right)=\frac{n}{(n,k)}$. In all cases, we have $\Ord\left(\frac{\alpha}{\beta}\right) \mid \lcm(4,n)$.

Equivalent analysis of the groups $S_4, A_4$ and $A_5$ reveals that in the case of exotic representations, we must have $\Ord\left(\frac{\alpha}{\beta}\right) \mid 60$. Together, with $D$ defined as in the lemma, we conclude that for any eigenform $f \in \snew{1}{N}{\chi}$ with Satake parameters $\alpha$ and $\beta$ at $p\nmid N$ we must have $\Ord\left(\frac{\alpha}{\beta}\right) \mid \lcm (60, D)$. Apply Lemma \ref{multdivlemma} to:
\begin{equation}
\Ord\left(\frac{\alpha}{\beta}\right) \mid \lcm (60, D),\;\;\; \Ord(\alpha \beta) \mid \Ord(\chi).
\end{equation} 
This gives $\Ord(\alpha) \mid 2\lcm(60,\Ord(\chi),D) \mid H$. For $p \mid N$ there is at most one Satake parameter, whose order must divide the order of the projective image, and therefore $H$.
\end{proof}

Consequently, if $d$ is the dimension of $\snew{1}{N}{\chi}$, there exists a linear transformation of $\tilde{V}$ such that the final $d$ columns are the projection onto $\mathbb{F}_q$ of an eigenbasis of $\snew{1}{N}{\chi}$. The next step is to diagonalise $\tilde{V}$ to acquire such a matrix.

Diagonalisation is performed as follows. First, we find a linear transformation of $\tilde{V}$ such that the leftmost columns form a basis for $\sold{1}{N}{\chi}$. The action of any Hecke operator $T$ on this matrix is of the form:
\begin{equation}
\left(
\begin{array}{c|c}
A^{\text{old}} & * \\ \hline
0 & A^{\text{new}} \\
\end{array}
\right),
\end{equation}
where $A^{\text{old}}$ is a square matrix representing the action of $T$ on $\sold{1}{N}{\chi}$. Given that $\tilde{V}$ has been computed up to at least the Sturm bound, there exists some linear combination of Hecke operators such that the eigenvalues of $A^{\text{new}}$ are unique and distinct from the eigenvalues of $A^{\text{old}}$. These eigenvalues are found by computing the roots of the characteristic polynomial of $A^{\text{new}}$ over $\mathbb{F}_q$. This polynomial may not split completely, indicating ethereal forms and or modular functions whose Fourier expansions lie in some extension of $\mathbb{F}_q$.

By taking just the linear factors of the characteristic polynomial, removing oldform contribution to the resulting eigenvectors, and normalising so that $a_n=1$, we construct a set of basis elements which contains the projection of an eigenbasis of $\snew{1}{N}{\chi}$ onto $\mathbb{F}_q$. In our computation up to level 10,000, the resulting space was in fact always equal to the projection of an eigenbasis of $\snew{1}{N}{\chi}$, but in theory it could still contain ethereal forms and modular functions.

The Satake parameters are subsequently found by factorising $x^2-a_px+\chi(p)$ for all primes $p$ (as in the diagonalisation step, any form for which we cannot compute all Satake parameters must be either an ethereal form or modular function, so can be discarded). These parameters are lifted back to $\mathbb{Z}[\zeta_H]$ by sending each element to the root of unity in its preimage in the original projection mapping $\sigma$. Hecke relations are then used to define all $a_n$ for all composite $n$ up to the Sturm bound.

\begin{lemma}
Diagonalising and lifting $\tilde{V}$ results in a matrix containing an eigenbasis for $\snew{1}{N}{\chi}$.
\end{lemma}

\begin{proof}
Let $f$ be a form in the eigenbasis for $\snew{1}{N}{\chi}$. From Lemma \ref{lemma:satakeorder}, we see that any Satake parameter $\alpha_p$ is in $\mathbb{Z}[\zeta_H]$, and so for any $p \mid N$ we have $a_p \in \mathbb{Z}[\zeta_H]$. For $p \nmid N$ we have $a_p = \alpha_p + \chi(p)\overline{\alpha_p}$ and so again $a_p \in \mathbb{Z}[\zeta_H]$. From Hecke relations, this is true for all $a_n$.

Consequently, the eigenvalues of projections of forms in $\snew{1}{N}{\chi}$ onto $\mathbb{F}_q$ will be present in the diagonalisation of $\tilde{V}$. The restriction of the original projection mapping $\sigma$ to the cyclic group generated by $\zeta_H$ is injective, and so recovery of the Fourier coefficients by lifting follows from Lemma \ref{lemma:satakeorder}.
\end{proof}

For any $f \in \snew{1}{N}{\chi}$, we have thus computed the projection of $f$ onto $\mathbb{F}_q$, and then recovered the original form via Satake parameters. However, it is still not guaranteed that all of the resulting functions are truncations of forms in $\snew{1}{N}{\chi}$. To verify the computation, we proceed as follows. Fix an Eisenstein series $E \in \ebasis{\chi}$, then:
\begin{equation}
f \in \snew{1}{N}{\chi} \iff f\overline{E} \in \sfull{2}{N}{\mathbbm{1}}.
\end{equation}
This is because $f$ is (by construction) Hecke stable, and (by Lemma \ref{lemma:eisintersect}) not an Eisenstein series, and so the characteristic 0 case of the Hecke stability method holds. This completes the proof of Theorem \ref{thm:MainTheorem}.

\section{Computation up to level $10,000$}
\label{sec:commentary}
The computation up to the $n$-th coefficient of $d$ basis elements of a $d$ dimensional cusp form space from its trace form is $O\left((nd)^{\frac{3}{2}}\right)$, as shown in \cite{Mypaper}. In order to perform the stabilisation step of our algorithm, we require $O(dp)$ coefficients of a basis for $\sfull{2}{N}{\mathbbm{1}}$ where $p$ is the least prime not dividing $N$. This is most efficiently performed by decomposing the space into the smallest possible subspaces; twists from twist-minimal spaces. For example, when computing the weight 1 forms at level $175$, we require $30$ coefficients of $15$ basis elements of $\sfull{2}{175}{\mathbbm{1}}$. This would involve computation up to the $330$-th coefficient of the trace form of $\snew{2}{175}{\mathbbm{1}}$, but only up to the $120$-th coefficient of the trace form of $\smin{2}{175}{\mathbbm{1}}$.

This computation gives us a methods of estimating the computation time up to a given level. By \cite[Theorem 8]{Martin}, the asymptotic expression for the dimension of $\sfull{2}{N}{\mathbbm{1}}$ is $O(N)$. Along with the Sturm bound, this gives the complexity of computing the initial matrix for $\sfull{2}{N}{\mathbbm{1}}$ as $O(N^3)$. Consequently, computing up to level $N_{\max}$ has complexity $O(N_{\max}^4)$. The time taken in our computations up to various levels suggests the estimate given by a degree 4 polynomial is reasonable. The following table is the time required for 192 cores to compute up to level $N_{\max}$ along with predicted time from fitting a degree 4 polynomial:

\begin{center}
\begin{tabular}{c | c | c}
$N_{\max}$ & Minutes to compute & Predicted minutes to compute \\ \hline
2,000 & 1 & 5 \\
3,000 & 7 & -3 \\
4,000 & 19 & 18 \\
5,000 & 53 & 64\\
6,000 & 130 & 136 \\
7,000 & 245 & 238 \\
8,000 & 398 & 377 \\
9,000 & 541 & 567\\
10,000 & 832 & 822\\
15,000 & - & 3,905\\
20,000 & - & 13,251\\
30,000 & - & 78,708\\
40,000 & - & 277,629\\

\end{tabular}
\end{center}

Dashes indicate that the level was not computed up to. Note that the anticipated computation time for $N_{\max}=20,000$ makes such a computation seem feasible. In practice, the main barrier we faced was not computation time but the memory requirements for finding a left pseudo-inverse of the weight 2 basis matrix over $\mathbb{Q}$. Such a computation is required for verifying $f \overline{E} \in \sfull{2}{N}{\mathbbm{1}}$ in the final step of the algorithm, but inverting matrices of approximate size $2,000 \times 2,000$ has a large memory overhead, such that even with 1TB of RAM we opted to perform these inversions one at a time past level $9,000$.

The final verification step involves a computationally expensive multiplication over a number field. The improvement in our approach comes from the fact that this only needs to be performed on a small number of elements (roughly the size of the exotic space). In contrast, methods which only use the characteristic $p$ computation as an upper bound would have to perform number field computations for all elements in weight 2. For example, at level $N=124$ we multiply two elements in $\snew{1}{N}{\chi_{124}(67,\cdot)}$\footnote{Here, $\chi_{124}(67,\cdot)$ refers to the Conrey label of the character, defined in \cite{Conrey12467}} rather than dividing fourteen elements in $\sfull{2}{N}{\mathbbm{1}}$.

To compute all spaces up to level $10,000$, we could simply use Theorem \ref{thm:MainTheorem} directly on every level and character up to this bound, but such a computation is improved by the following result reducing the scope of the search.

\begin{lemma}
\label{restrictionlemma}
Let $p \mid N$ be any prime satisfying $\nu_p(N) \in \{1,\nu_p(\mathfrak{f}(\chi))\}$. If $\chi_p=\mathbbm{1}$ then $\snew{1}{N}{\chi}=0$, and if $\Ord(\chi_p)>5$ then $\mathcal{S}_1^{\text{exotic}}(N,\chi)=0$. Further, if $p$ and $q$ are two such primes, with $\Ord(\chi_p \chi_q)=20$, then $\mathcal{S}_1^{\text{exotic}}(N,\chi)=0$.
\end{lemma}
\begin{proof}
The results on $\mathcal{S}_1^{\text{exotic}}$ are given in \cite{BookerLeeStrombergsson}. The additional result on all of $\mathcal{S}^{\text{new}}$ comes from the remarks after (\ref{apis0}).
\end{proof}
In general, for a level $N$ and character $\chi$, we only need to compute $\snew{1}{N}{\chi}$ if it is possible that $\mathcal{S}_1^{\text{exotic}}(N,\chi)$ is non-zero, or if there exists some multiple $M$ of $N$ which we wish to compute where it is possible that $\mathcal{S}_1^{\text{exotic}}(M,\chi)$ is non-zero, such that we will require lifts of the elements of $\mathcal{S}_1^{\text{dihedral}}(N,\chi)$ as oldforms. This requirement drives Definition \ref{admissible} and Lemma \ref{admissiblelemma} follows trivially.

\bibliography{papers}
\bibliographystyle{plain}

\end{document}